\documentclass[twoside,a4paper,reqno,11pt]{amsart} 
\usepackage{amsfonts, amsbsy, amsmath, amssymb, latexsym}
\usepackage{mathrsfs,array}                
\usepackage[top=30mm,right=30mm,bottom=30mm,left=30mm]{geometry}
\usepackage{bm}
\usepackage{graphicx}
\usepackage{amsmath, amssymb, amsfonts, amscd, amsthm, mathrsfs, mathtools}
\usepackage[colorlinks, linkcolor= blue, citecolor= red]{hyperref}
\usepackage{comment}
\usepackage{color}

\newcommand{\SL}{\mathrm{SL}}
\newcommand{\SU}{\mathrm{SU}}

\newcommand{\Sp}{\mathrm{Sp}}
\newcommand{\PSp}{\mathrm{PSp}}

\newcommand{\PSL}{\mathrm{PSL}}
\newcommand{\PSU}{\mathrm{PSU}}

 \newcommand{\widesim}{
  \mathrel{{\scalebox{1.5}[1]{$\sim$}}}
}

\newcommand{\reducesize}[2]{%
  \mathbin{
    \ooalign{
      \raisebox
       {.4ex}
          {$#1\widesim$}
      \cr 
      \hidewidth
      \raisebox
        {-.6ex}
        {\scalebox
          {.75}
          {$#1#2$}
        }
      \hidewidth
    }
  }
}

\newcommand{\stb}[1]{\mathpalette\reducesize{#1}}

\newcommand{\CC}{\mathcal{C}}
\newcommand{\DD}{\mathcal{D}}

\newcommand{\Fq}{\mathbb{F}_q}

            \def\e{\epsilon} 
            \def\la{\langle} 
            \def\ra{\rangle} 
          
            \def\a{\alpha} 
             \def\b{\beta}

            \def\e{\epsilon} 
            \def\O{\Omega} 
             
            \def\l{\lambda}

            \def\F{\mathbb{F}}

            \def\d{\delta} 
     \def\D{\Delta}

     \def\G{\Gamma} 

     \def\no{\noindent}

\newtheoremstyle{pedrodef}{}{}{}{}{\bfseries}{.}{ }{\thmname{#1}\thmnumber{ #2}\thmnote{ (#3)}}

\newtheorem{theorem}{Theorem} 
 
\newtheorem{thm}{Theorem}[section] 
\newtheorem{prop}[thm]{Proposition} 
\newtheorem{lem}[thm]{Lemma}

\theoremstyle{pedrodef}
\newtheorem{rem}[thm]{Remark} 
\newtheorem{defn}[thm]{Definition} 
\newtheorem{ex}[thm]{Example} 

\begin{document}

           \title{The binary actions of simple groups of Lie type of characteristic 2}

\date{\today}

\begin{abstract}
Let $\mathcal{C}$ be a conjugacy class  of involutions in a group $G$. We study the graph $\Gamma(\mathcal{C})$ whose vertices are elements of $\mathcal{C}$ with $g,h\in\mathcal{C}$ connected by an edge if and only if $gh\in\mathcal{C}$. For $t\in \mathcal{C}$, we define the \emph{component group} of $t$ to be the subgroup of $G$ generated by all vertices in $\Gamma(\mathcal{C})$ that lie in the connected component of the graph that contains $t$.

We classify the component groups of all involutions in simple groups of Lie type over a field of characteristic $2$. We use this classification to partially classify the transitive binary actions of the simple groups of Lie type over a field of characteristic $2$ for which a point stabilizer has even order. The classification is complete unless the simple group in question is a symplectic or unitary group.
\end{abstract}
\author{Nick Gill}
\address{School of Mathematics and Statistics, The Open University, Walton Hall, Milton Keynes, MK7 6AA, UK}
\email{nick.gill@open.ac.uk}

\author{Pierre Guillot}
\address{IRMA, 7 rue René Descartes, 67084 Strasbourg, France}
\email{guillot@math.unistra.fr}

\author{Martin W. Liebeck}
\address{Department of Mathematics, Imperial College London, London, SW7 2AZ, UK}
\email{m.liebeck@imperial.ac.uk }

\maketitle

\section{Introduction}

Let $G$ be a finite group acting on a finite set $\Omega$. Let $I, J \in \Omega^n$ be $n$-tuples of elements of $\Omega$, for some $n \ge 2$, written $I=(I_1, \ldots, I_n)$ and $J= (J_1, \ldots, J_n)$. For $r \le n$, we say that $I$ and $J$ are {\em $r$-related}, and we write $I \stb{r} J$, when for each choice of indices $1 \le k_1 < k_2 < \cdots < k_r \le n$, there exists $g \in G$ such that $I_{k_i}^g = J_{k_i}$ for all $i$.

We say that the action of $G$ on $\Omega$ is \emph{binary} if, for all $n\ge 2$ and all $I,J\in\Omega^n$,  $I\stb{2} J$ implies that $I\stb{n} J$. The study of binary actions is motivated by considerations in model theory \cite{cherlin_martin, cherlin1}.

In recent years there has been substantial progress in our understanding of binary actions. We have a full classification of the primitive binary permutation groups \cite{gls_binary}. We have also started to study the binary actions of the simple groups; in particular, we now have a full classification of the binary actions of the alternating groups \cite{ggAltBin, ggLieBin}.

The first aim of this paper is to contribute to the classification of the binary actions of the simple groups by proving a result for simple groups of Lie type over a field of characteristic 2. Our main result in this direction is the following.

\begin{theorem}\label{t: char 2}
 Let $G = G(q)$ be a simple group of Lie type over $\F_q$, where $q = 2^a$. Suppose that $G$ has a binary action on a set $\Omega$ and that there exists $\omega\in\Omega$ such that $G_\omega$, the stabilizer in $G$ of $\omega$, is a proper subgroup of $G$ of even order. Then one of the following holds.
 \begin{itemize}
  \item[{\rm (i)}] $G={^2B_2}(2^{2a+1})$ with $a\geq 1$ and $G_\omega$ is the centre of a Sylow $2$-subgroup of $G$; 
  \item[{\rm (ii)}] $G=\Sp_{2n}(2^a)$ or $\PSU_{n}(2^a),$ with $n\geq 2$, and $G_\omega$ contains the centre of a long root subgroup of $G$;
  \item[{\rm (iii)}] $G=\Sp_{4}(2^a)$ with $a\geq 2$ and $G_\omega$ contains a short root subgroup of $G$.
 \end{itemize}
\end{theorem}

Note that the theorem applies with $G$ equal to $\Sp_4(2)'$, $G_2(2)'$ and $^2\!F_4(2)'$. Note too that the theorem applies to all actions of $G$ -- there is no assumption of transitivity, for instance. Let us briefly discuss the three listed items in the theorem.

In the first item, $G$ is a Suzuki group and the transitive binary actions of $G$ are completely classified \cite{ggLieBin}. In this case we know that the action of $G$ on the set of right cosets of a non-trivial proper subgroup $H$ is binary if and only $H$ is the centre of a Sylow $2$-subgroup of $G$.

For the second and third items, our expectation is that the result would remain true were the word ``contains'' to be replaced, in each case, with the word ``is''. This would amount to a classification of the non-trivial transitive binary actions of the simple groups of Lie type over a field of characteristic 2 for which a point-stabilizer has even order.

To emphasise this point we present a result at the end of the paper -- Proposition~\ref{p: pc} -- which implies that, were we to make the aforementioned replacement and restrict our attention to transitive actions, then the resulting list of actions would all be binary.

Our proof of Theorem~\ref{t: char 2} makes use of methods introduced in \cite{ggAltBin}, in particular the notion of a \emph{component group}. Our second main result pertains to a special case of this notion which we now define (more detail is given in \S\ref{s: background}).

Let $G$ be a finite group, let $t$ be an involution in $G$ and let $\mathcal C = t^G$ be the conjugacy class of $t$. We define a graph $\Gamma(\mathcal{C})$, whose vertices are elements of $\mathcal{C}$; two vertices $g,h\in \mathcal{C}$ are connected by an edge in $\Gamma(\mathcal{C})$ if $gh\in\mathcal{C}$. We write $\Delta(t)$ for the group generated by all vertices in $\Gamma(\mathcal{C})$ that lie in the connected component of the graph that contains $t$. We call $\Delta(t)$ the \emph{component group} of the element $t$.

We are ready to state our second main result. For this result we do not include $\Sp_4(2)'$ or $G_2(2)'$ in the hypothesis (but we do include $^2\!F_4(2)'$). Note also that the group $G=\SL_2(q)$ is included under item (ii) of the theorem, since $\Sp_2(q) = \SL_2(q)$. (The same comment applies to part (ii) of Theorem \ref{t: char 2}.)

\begin{theorem}\label{t: component} Let $G = G(q)$ be a simple group of Lie type over $\F_q$, where $q = 2^a$, and let $t \in G$ be an involution. Then one of the following holds:
\begin{itemize}
\item[{\rm (i)}] $\D(t) = G$;
\item[{\rm (ii)}] $t$ is a long root involution, $G = \Sp_{2n}(q)$, $\PSU_n(q)$ or $^2\!B_2(q)$, and $\D(t)$ is the centre of a long root subgroup (the centre of a Sylow $2$-subgroup when $G=\,^2\!B_2(q)$);
\item[{\rm (iii)}] $t$ is a short root involution, $G = \Sp_{4}(q)$ $(q>2)$, and $\D(t)$ is a short root subgroup;
\item[{\rm (iv)}] $G, t$ and $\D(t)$ are as in Table $\ref{tbl}$.
\end{itemize}
\end{theorem}

\begin{table}[h!]
\label{tbl}
 \caption{Exceptional involution classes}
\begin{tabular}{|c|c|c|}
\hline
$G$ & $t$ & $\D(t)$ \\
\hline &&\\[-1em]
$\SL_6(2)$ & $(J_2^{\,3})$ & $2^9$ \\
$\PSU_6(2)$ & $(J_2^{\,3})$ & $2^9$ \\
$\PSU_7(2)$ & $(J_2^{\,3},J_1)$ & $2^9$ \\
$\Sp_6(2)$ & $W(2)+V(2)$ & $2^6$ \\
$\O_{12}^+(2)$ & $W(2)^3$ & $2^{15}$ \\
$\Sp_{12}(2)$ & $W(2)^3$ & $2^{15}$\\
\hline
\end{tabular}
\end{table}

The notation for the classes of $\Sp_6(2)$ and $\O_{12}^+(2)$ in Table \ref{tbl} is described in  \S \ref{orth}. Each line of the table corresponds to a single conjugacy class in $G$, except for the line for $G = \O_{12}^+(2)$, where $W(2)^3$ corresponds to two $G$-classes, interchanged by an outer automorphism. We remark that for the classes in Table \ref{tbl}, $\D(t)$ is elementary abelian, and in all cases except $G = \PSU_7(2)$, we have $\D(t) = O_2(C_G(t))$.

\subsection*{Structure of the paper}

In \S\ref{s: background} we discuss the graph $\Gamma(\mathcal{C})$ in a more general context than that given above and we state a number of results that connect this graph to the binary actions.

In \S\ref{s: computing} we describe some computational methods that pertain to the graph $\Gamma(\mathcal{C})$; this section concludes with several lemmas that will be important for the proofs of Theorems~\ref{t: char 2} and \ref{t: component}. In \S\ref{s: proof component} we give the proof of Theorems~\ref{t: component} and in \S\ref{s: proof char 2} we give the (very short) proof of Theorem~\ref{t: char 2}. In the final section we prove a proposition which implies a partial converse to Theorem~\ref{t: char 2}.

\section{Binary actions and component groups}\label{s: background}

All groups mentioned in this paper are finite, and all group actions are on finite sets.

\subsection{Component groups} \label{subsec-component-groups}

We mentioned (a special case of) the graph $\Gamma(\mathcal{C})$ in the introduction. This graph was first defined in \cite{ggAltBin} and we give the definition here.

\begin{defn}\label{d: graph}
Given a conjugacy class $\mathcal{C}$ in a group $G$ we define a graph, $\Gamma(\mathcal{C})$, whose vertices are elements of $\mathcal{C}$; two vertices $g,h\in \mathcal{C}$ are connected in $\Gamma(\mathcal{C})$ if $g$ and $h$ commute and either $gh^{-1}$ or $hg^{-1}$ is in $\mathcal{C}$.
 \end{defn}
 
The next lemma connects the graph $\Gamma(\mathcal{C})$ to the notion of a binary action; this lemma first appeared as \cite[Corollary~2.16]{ggAltBin}. The \emph{fixity} of an element in group acting on a set $\Omega$ is the number of points of $\Omega$ fixed by the element.
 
\begin{lem}\label{l: graph}
Let $G$ act transitively on $\Omega$, let $H$ be the stabilizer of a point in $\Omega$, let $p$ be a prime dividing $|H|$, let $\CC$ be a conjugacy class of elements of order $p$ of maximal fixity and let $g$ be in $H\cap \mathcal{C}$. If the action is binary, then $H$ contains all vertices in the connnected component of $\Gamma(\CC)$ that contains $g$.
\end{lem}
 
In what follows we will write ``element of maximal $p$-fixity'' as shorthand for ``element of prime order $p$ of maximal fixity''. 
 
In the notation of Lemma~\ref{l: graph}, we say that the \emph{component group} of $g$ in $G$ is the group generated by the connected component of $\Gamma(\CC)$ that contains $g$; we write this as $\Delta(g)$. So the conclusion of Lemma~\ref{l: graph} could be ``$H$ contains $\Delta(g)$''.

It turns out that Definition~\ref{d: graph} and Lemma~\ref{l: graph} can be stated slightly more generally. To do this we need to revisit an example from \cite{ggAltBin}, on which the proof of Lemma~\ref{l: graph} is based. Here and below, we write $Fix(g)$ to denote the fixed set of a permutation $g$.

\begin{ex}\label{ex: well known}
Let $p$ be a prime and let $g_1, g_2$ be distinct commuting elements of order $p$ in a group $G$ acting on a set $\Lambda$. Set $g_3=g_1g_2^{-1}$ and write $F_i$ for the fixed set of $g_i$ for $i=1,2,3$. Assume that
\[
 |F_1|=|F_2|=|F_3|\geq 1
\]
and assume that $F_1$ and $F_2$ are distinct (which, in turn, means that $F_3$ is distinct from both $F_1$ and $F_2$). 

Then $\langle g_1, g_2\rangle$ acts on the set $F=F_1\cup F_2\cup F_3$ and we write $\tau_0$ for the (non-trivial) permutation of $F_3$ induced by $g_1$ (and $g_2$) on the set $F_3$. Finally let $\tau$ be the permutation of $F$ which equals $\tau_0$ on $F_3$ and fixes all points in $F\setminus F_3$. 

Write $I=(f_1,\dots, f_k)$, where $F=\{f_1,\dots, f_k\}$ and $k=|F|$, and write $J=(f_1^\tau,\dots, f_k^\tau)$. It is easy to verify directly that
\begin{enumerate}
 \item $I\stb{2} J$;
 \item $I\stb{k} J$ if and only there exists a permutation $h\in G$ that fixes $F$ setwise and that induces the permutation $\tau$ on $F$.
\end{enumerate}
In particular, if the action of $G$ on $\Lambda$ is binary, then there exists a permutation $h\in G$ that fixes $F$ setwise and that induces the permutation $\tau$ on $F$. Note that $Fix(h)$ properly contains $Fix(g_i)$ for $i=1,2$.
\end{ex}

\begin{defn}\label{d: graph2}
 Given a union of conjugacy classes $\mathcal{D}$ in a group $G$ we define a graph, $\Gamma(\mathcal{D})$, whose vertices are elements of $\mathcal{D}$; two vertices $g,h\in \mathcal{D}$ are connected in $\Gamma(\mathcal{D})$ if $g$ and $h$ commute and either $gh^{-1}$ or $hg^{-1}$ is in $\mathcal{D}$.
 \end{defn}

For $g\in \mathcal{D}$, we write $\Delta(g,\mathcal{D})$ for the group generated by the vertices in the connnected component of $\Gamma(\mathcal{D})$ that contains $g$. In particular, if $\mathcal{C}$ is the conjugacy class containing $g$, then $\Delta(g,\CC)=\Delta(g)$, the component group of $g$.

 \begin{lem}\label{l: graph2}
Let $G$ act transitively on $\Omega$, let $H$ be the stabilizer of a point in $\Omega$, let $p$ be a prime dividing $|H|$, let $\mathcal{D}$ be a union of conjugacy classes of elements of maximal $p$-fixity and let $g$ be in $H\cap \mathcal{D}$. If the action is binary, then $\Delta(g,\mathcal{D})\leq H$.
\end{lem}
\begin{proof}
Assume, for a contradiction, that $\Delta(g,\mathcal{D})\not\leq H$. Then there exist $g_1, g_2\in \mathcal{D}$ such that $g_1$ and $g_2$ are adjacent in $\Gamma(\mathcal{D})$, $g_1\in H$ and $g_2\not\in H$. This immediately implies that $g_1$ and $g_2$ commute and that $Fix(g_1)$ is distinct from $Fix(g_2)$. Now the set-up of Example~\ref{ex: well known} applies. Since the action of $G$ on $\Omega$ is assumed to be binary, there must exist an element $h$ whose fixed set properly contains $Fix(g_1)$. This contradicts the fact that $g_1$ is an element of maximal $p$-fixity and we are done.
\end{proof}

\subsection{Terminal component groups}

The next lemma inspires the definitions that follow.

\begin{lem}\label{l: fix comp}
Let $G$ act transitively on $\Omega$, let $H$ be the stabilizer of a point in $\Omega$, let $p$ be a prime dividing $|H|$, let $\mathcal{D}$ be a union of conjugacy classes of elements of maximal $p$-fixity and let $g$ be in $H\cap \mathcal{D}$. If the action is binary, then every element of $\Delta(g,\mathcal{D})$ that has order $p$ is also an element of maximal $p$-fixity.
\end{lem}
\begin{proof}
By assumption, the action of $G$ is binary and so Lemma~\ref{l: graph2} implies that the stabilizer of any point fixed by $g$ contains $\Delta(g,\mathcal{D})$. Thus every element of $\Delta(g,\mathcal{D})$ fixes at least as many elements of $\Omega$ as $g$.
\end{proof}

Now consider a finite group $G$ and an element $g$ contained in $\DD$, a union of conjugacy classes of elements of order $p$. Let $\DD_1=\DD$ and let $\Gamma_1(\DD)=\Gamma(\DD)$ and $\Delta_1(g)=\Delta(g,\DD)$.

For a positive integer $i$, define $\DD_i$ to be the union of conjugacy classes in $G$ which satisfy two criteria:
\begin{enumerate}
 \item elements of these conjugacy classes have order $p$;
 \item for any such conjugacy class, $C$, we have $C\cap \Delta(g, \DD_{i-1})\neq\emptyset$
\end{enumerate}

Now define $\Delta_i(g,\DD)=\Delta(g,\DD_{i})$. We have an ascending chain of subgroups:
\[
 \Delta_1(g,\DD) \leq \Delta_2(g,\DD) \leq \Delta_3(g,\DD)\leq\cdots
\]
We define $\Delta_\infty(g,\DD)$ to be the union of all of the subgroups in this chain. In the case where $\DD=\CC$, the conjugacy class containing $g$, we write $\Delta_\infty(g)=\Delta_\infty(g,\DD)$ and call this group the \emph{terminal component group of $g$ in $G$}. This definition allows us to strengthen Lemma~\ref{l: fix comp}:

\begin{lem}\label{l: terminal component}
Let $G$ act transitively on $\Omega$, let $H$ be the stabilizer of a point in $\Omega$, let $p$ be a prime dividing $|H|$, let $\mathcal{D}$ be a union of conjugacy classes of elements of maximal $p$-fixity and let $g$ be in $H\cap \mathcal{D}$. If the action is binary, then $H$ contains $\Delta_\infty(g,\DD)$.

In particular if $g$ is any element of maximal $p$-fixity, then $H$ contains $\Delta_\infty(g)$, the terminal component group of $g$.

\end{lem}
\begin{proof}
We suppose that $\DD_i$ is a union of conjugacy classes of elements of maximal $p$-fixity. We will use induction to prove that, for all positive integers $i$,
\begin{enumerate}
 \item[(S1)] $\Delta_i(g,\DD)\leq H$; and
 \item[(S2)] $\DD_{i+1}$ is a union of conjugacy classes of elements of maximal $p$-fixity.
\end{enumerate}

For $i=1$, (S1) follows from Lemma~\ref{l: graph2} and (S2) follows from Lemma~\ref{l: fix comp}. Thus the base case of our inductive argument is proved.

Let us assume that (S1) and (S2) are true for $i$. Since (S2) is true for $i$, Lemma~\ref{l: graph2} implies that (S1) is true for $i+1$. Now Lemma~\ref{l: fix comp} implies that (S2) is true for $i+1$. The result follows by induction.
\end{proof}

\begin{rem}
In this paper we will interested primarily in the groups $\Delta(g)$ and $\Delta_\infty(g)$ for various elements $g$ of prime order in $G$.

The most obvious scenario where the more general notion of $\Delta(g,\DD)$ and $\Delta_\infty(g,\DD)$ is of interest would be setting $\mathcal{D}$ to be the \emph{rational conjugacy class} containing $g$. It is clear that in any action of $G$, if $g$ is an element of maximal $p$-fixity, then the same will be true of all of the rational conjugates of $g$.
\end{rem}

\section{Computing with component groups}\label{s: computing}

The proof of Theorem \ref{t: component} involves the calculation of several component groups in certain exceptional cases, using a computer. We have relied on GAP and Sagemath, but it is not possible to use a naive approach for the task at hand. In this section, we describe how the computation was made possible. Besides, the ``transport group", to be defined shortly, is probably a useful object to think of, even in theoretical situations.

In what follows, $G$ is a finite group, $s \in G$ is an involution and $\CC$ is its conjugacy class. We wish to compute the connected component $X$ of $\Gamma(\CC)$ containing $s$; the subgroup generated by this connected component is then determined readily by standard algorithms. 

Note that the techniques we describe would work equally well with $s$ an element of prime order $p$, although we stick to $p=2$ for simplicity of language (and because this is what the paper is about).

\subsection{Introducing the transport group}

The naive approach to our problem would consist in finding all the neighbours of $s$ in the graph $\Gamma(\CC)$, by scanning all the elements of $\CC$ one by one ; then finding the neighbours of these neighbours ; and then iterating until no new vertex can be appended to the connected component. This is way too slow in practice. 

We will mostly reduce things to the computation of the original neighbours of $s$, with no need to dive deeper in the graph. An essential ingredient is the following lemma :

\begin{lem}
Let $Y$ be a finite, connected graph, and let $T$ be a subgroup of $Aut(Y)$ with the following property: there is a vertex $v$ such that all its neighbours are of the form $v^t$ for $t \in T$. Then the action of $T$ is vertex-transitive.
\end{lem}

\begin{proof}
Let $w \in Y$. To show that $w \in v^T$, we proceed by induction on the length $n$ of a minimal path from $v$ to $w$ in the graph, the case $n=1$ being given. We see immediately that $w$ has a neighbour $u$ such that $u \in v^T$, by induction. If $u = v^t$ for $t \in T$, then $w^{t^{-1}}$ is a neighbour of $v$, so that $w^{t^{-1}} = v^{t'}$ for some $t' \in T$, and $w = v^{t't}$ as desired.
\end{proof}

We will apply this with $Y=X$, the connected component of $s$, viewed as an induced subgraph of $\Gamma(\CC)$, and of course with $v=s$. By definition, all the vertices of $\Gamma(\CC)$ are conjugates of $s$ within $G$, so that we may find elements $t_i \in G$ such that the neighbours of $s$ can be computed to be of the form $s^{t_1}, \ldots, s^{t_d}$. Each $t_i$ acts on $\Gamma(\CC)$ by graph automorphisms, and takes $s$ to an element of $X$, and so it preserves the connected subgraph~$X$. We can then apply the lemma to the subgroup $T= \langle t_1, \ldots, t_d \rangle$, and we find that $X= s^T$.

This provides already a comfortable speedup, but we also need to select the group $T$ more carefully, as finding the neighbours of $v$ is in itself a costly operation. Since these neighbours must commute with $s$ by definition, we elect to: \begin{enumerate}
\item compute first the centralizer $Z = C_G(s)$ ;
\item determine the conjugacy classes of involutions in $Z$ ;
\item for each such class, pick one representative $x$, and check whether there is $t \in G$ such that $s^t = x$ (so that $x \in \CC$), and then check whether $sx \in \CC$; if not, discard the conjugacy class;
\item finally, define the {\em transport group} $T$ to be generated by $Z$ and all the elements $t$ obtained during the previous step.
\end{enumerate}

The lemma then applies with this $T$, as is checked readily (simply note that in step (3), we only need one representative, as there is an element of $Z$ (which is thus in $T$) which fixes $s$ and takes $x$ to any other involution in the conjugacy class).

Having gone through these steps, it is very easy to compute the order of $T$, and $|T|/|Z|$ is the size of $s^T = X$. Very often, we find that $X= \Gamma(\CC)$ so that, when $G$ is simple, we can conclude that $\Delta(s) = G$. Otherwise, standard algorithms allow to iterate over the elements of $X$ and compute the group they generate.

\begin{rem}
To give a bit of perspective, let us mention the sort of performance we obtained with $G= \Omega_{12}^+(2)$ for some randomly chosen involution $s$. The naive algorithm took about 8 minutes to compute the neighbours of $s$, and there are around 800 of these, so we expected to have to wait 4 days to get the list of the vertices in $X$ which are at distance $2$ from $s$. By contrast, the algorithm above concludes that $\Delta(s)=G$ in less than 20 seconds.
\end{rem}

\subsection{Randomized calculations}

The above algorithm can still take a lot of time in certain cases (for example when the centralizer $Z$ has a complicated structure). However, another advantage of the approach we have just described is that it can be turned easily into a randomized algorithm. The idea is simply to draw elements of $\CC$ at random until one finds a neighbour of $s$ in $\Gamma(\CC)$; then, compute $t \in G$ such that the vertex just found is $s^t$, and keep a list of the elements $t$ thus obtained. At any moment, we can compute the group $T$ generated by $C_G(s)$ and the $t$'s that have been collected. If at any point $T = G$ (which happens frequently with our examples), of course we may stop and conclude that $\Delta(s) = G$ (when $G$ is simple). Otherwise, we can stop after a certain number of tries, and what we have is a subgroup $T$ of the "real" transport group, and we can compute the group $\Delta'(s)$ generated by $s^T$, which is a subgroup of the "real" component group $\Delta(s)$.

In this situation, we can try to compute the terminal component group $\Delta_\infty(s)$ instead, as in \S\ref{subsec-component-groups}.  Concretely, we check whether $\Delta'(s)$ (and so also $\Delta(s)$) contains an involution $s_1$ such that $s_1 \not\in \CC$ ; if there is one, we can run the algorithm with $s_1$, obtaining $\Delta'(s_1)$, a subgroup of $\Delta(s_1)$, and crucially, $\Delta'(s_1)$ is also a subgroup of $\Delta_\infty(s)$, by definition. We can continue and find $s_2$ in either $\Delta'(s)$ or $\Delta'(s_1)$, compute $\Delta'(s_2)$, which is yet another subgroup of $\Delta_\infty(s)$, and so on. We stop when there are no more elements to try (because all of their conjugacy classes have been attempted), or when there is $s_i$ such that $\Delta'(s_i) = G$, in which case we conclude that $\Delta_\infty(s) = G$ also.

This search is best organized, and more easily summarized, as follows. Our program constructs a graph whose vertices are the conjugacy classes of involutions in $G$. When we have found that $\Delta(s) = G$ for an involution $s$, we paint the corresponding vertex black. Otherwise, the vertex is left in white, and we add a directed arrow from the class of $s$ to that of any $s_1$ such that $s_1 \in \Delta'(s)$, our computed approximation of $\Delta(s)$. A sufficient condition for $\Delta_\infty(s) = G$ is thus that from the corresponding vertex, there is a directed path leading to a black vertex.

As an example, here is a graph produced when working with $G= \Omega_{12}^+(2)$. It shows that $\Delta_\infty(s) = G$ for any involution $s$.

\begin{center}
 \includegraphics[width=.5\textwidth]{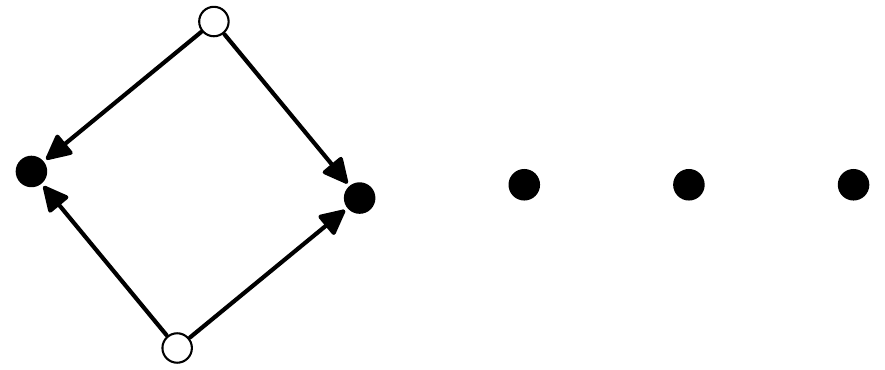}
\end{center}

In general, the graphs thus produced only reflect partial information, as obtained at a given point in time by the randomized calculation. Typically, we then focus our efforts on the white vertices, and start the full, non-random algorithm on them, to see if $\Delta(s)$ is really smaller than $G$. Of course, this may take a little longer.

\subsection{Results}

To conclude this section we record three results that will be important for the proof of Theorem~\ref{t: component}. These were all proved using the computational methods described above.

\begin{lem}\label{l: small classical} Let $G$ be one of the groups $\SL_6(2)$, $\PSU_n(2)$ ($4\le n\le 8$), $\Sp_{2n}(2)$ ($3\le n\le 6$), 
$\O_{2n}^\pm(2)$ ($4\le n\le 6$), and let $t \in G$ be an involution. Then one of the following holds:
\begin{itemize}
\item[{\rm (i)}] $\D(t) = G$;
\item[{\rm (ii)}]  $G = \PSU_n(2)$ or $\Sp_{2n}(2)$ and $t$ is a long root involution;
\item[{\rm (iii)}] $G$, $t$ are as in Table $\ref{tbl}$ of Theorem $\ref{t: component}$.
\end{itemize}
\end{lem}

\begin{lem}\label{l: small exceptional}
 Let $G$ be equal to $^2\!F_4(2)'$ or $^3\!D_4(2)$ and let $t\in G$ be an involution. Then $\Delta(t)=G$.
\end{lem}

\begin{lem}\label{l: terminal table}
 Let $G$ be a group occurring in Table~\ref{tbl} and let $t$ be the associated involution. Then $\Delta_\infty(t)=G$.
\end{lem}

Notice that Lemma~\ref{l: terminal table} implies that if $G=G(q)$ is a simple group of Lie type over $\Fq$ where $q=2^a$ and $t\in G$ is an involution, then either $\Delta_\infty(t)=G$ or else $G$ and $t$ are as described at points (ii) and (iii) of Theorem~\ref{t: component}.

\section{Proof of Theorem \ref{t: component}}\label{s: proof component}

\subsection{Long root elements}

Here we prove a lemma identifying component groups of long root elements for arbitrary characteristic. 
In the statement we do not include $G = \Sp_4(2)',\,G_2(2)'$ or $^2\!G_2(3)'$, but we do include $^2\!F_4(2)'$.

\begin{lem}\label{rooty} Let $G = G(q)$ be a simple group of Lie type over $\F_q$, where $q = p^a$, and let $t$ be a long root element. Then one of the following holds:
\begin{itemize}
\item[{\rm (i)}] $\D(t) = G$;
\item[{\rm (ii)}] $G = \PSp_{2n}(q)\,(n\ge 1)$ or $\PSU_n(q)\,(n\ge 3)$, and $\D(t)$ is the centre of a long root subgroup;
\item[{\rm (iii)}]  $G = \,^2\!G_2(q)$ or $^2\!B_2(q)$, and $\D(t)$ is the centre of a Sylow $p$-subgroup of $G$.
\end{itemize}
\end{lem}

Proposition~\ref{p: pc}, proved at the end of the paper, asserts that if $G$ is one of the groups listed in items (ii) or (iii), $t$ is a long root element and $H=\D(t)$, then either $G=\PSp_2(q)$ with $q$ odd or the action of $G$ on the set of right cosets of $H$ in $G$ is binary.

\begin{proof} First assume that $G$ is not one of the groups under (ii) or (iii), and is of untwisted type. Let $\Phi$ be the root sytem of $G$, and $\Phi_L$ the set of long roots in $\Phi$ (so $\Phi_L = \Phi$ if there is only one root length). We can take $t = u_\a(1)$ with $\a \in \Phi_L$. As $G \ne C_n(q)$, there exists $\b \in \Phi_L$ such that $\a+\b \in \Phi_L$, and $\la U_{\pm \a}, U_{\pm \b} \ra = A_2(q)$. Working in $\SL_3(q)$ with the usual representation
\[
u_\a(c) = \begin{pmatrix} 1&c& \\&1& \\&&1 \end{pmatrix}, 
u_\b(c) = \begin{pmatrix} 1&& \\&1&c \\&&1 \end{pmatrix}, 
u_{\a+\b}(c) = \begin{pmatrix} 1&&c \\&1& \\&&1 \end{pmatrix}, 
\]
we see that the graph $\G(\mathcal C)$ has the following edges, for any $c_i \in \F_q\setminus 0$:
\[
u_{-\b}(c_1)\; \mbox{---}\; u_\a(c_2) \; \mbox{---}\; u_{\a+\b}(c_3) \; \mbox{---}\; u_\b(c_4) \; \mbox{---}\; u_{-\a}(c_5)
\]
Hence $\D(t)$ contains $\la U_{\pm \a}, U_{\pm \b} \ra = A_2(q)$. 

Now take $\a = -\a_0$, where $\a_0$ is the highest root. Then for any long root $\rho$ such that $\b = \a_0-\rho \in \Phi_L$, we see from the above that $\D(t)$ contains $U_{\pm \b}$. If there is only one root length, these root groups $U_{\pm \b}$, together with $U_{\pm \a}$, generate $G$; so $\D(t) = G$ in this case. Finally, suppose there are two root lengths. Then $\Phi$ is of type $B_n\,(n\ge 3)$, $G_2$ or $F_4$, and the root groups $U_{\pm \b}$ for $\b \in \Phi_L$ generate a subsystem subgroup $D_{n-1}(q)$, $A_2(q)$ or $D_4(q)$ respectively. So $\D(t)$ contains this subsystem subgroup, which is maximal in $G$. Also $\D(t)$ is invariant under $C_G(t)$, which is the derived group of a parabolic subgroup of $G$, and is not contained in the subsystem subgroup. It follows that $\D(t) = G$ in this case also.

Now assume that $G$ is of twisted type, and is not one of the groups under (ii) or (iii). For $G = \,^2\!D_n(q)\,(n\ge 4)$, $^3\!D_4(q)$  or $^2\!E_6(q)$, the above argument gives the result, as there is a subsystem $A_2$ spanned by long roots. And for $G = \,^2\!F_4(q)'$, we argue as follows. The involution $t$ lies in a subgroup $^2\!F_4(2)'$ of $G$, and Lemma \ref{l: small exceptional} implies that $\D(t) \ge \,^2\!F_4(2)'$. As $\D(t)$ is $C_G(t)$-invariant, it follows that $\D(t) = G$.

Now we let $G$ be one of the group under (ii) and (iii) and we must describe $\Delta(t)$. We start by assuming that $G = \PSp_{2n}(q)$. With respect to a suitable standard basis $e_1,\ldots ,e_n,f_n,\ldots, f_1$, modulo the scalars $Z = \la -I\ra$, we can take
\begin{equation}\label{rep}
t = \begin{pmatrix} 1&&\l \\&I_{2n-2}& \\&&1 \end{pmatrix},
\end{equation}
for some scalar $\l \ne 0$, and then 
\[
C_G(t) = \left\{ \begin{pmatrix} \e &x&c \\&A&AJx^T \\&&\e \end{pmatrix} : A \in \Sp_{2n-2}(q),\,x \in \F_q^{2n-2},\,c \in\F_q,\,\e = \pm 1\right\}/Z,
\]
where $J$ is the matrix of the form restricted to $e_2,\ldots ,e_n,f_n,\ldots, f_2$. The only elements $u \in C_G(t)$ for which 
both $u$ and $tu^{-1}$ are conjugate to $t$ are those with $A=I_{2n-2}$ and $x = 0$. Hence 
\begin{equation}\label{ugp}
u \in \left\{ \begin{pmatrix} 1&&c \\&I_{2n-2}& \\&&1 \end{pmatrix} : c \in \F_q\right\} = U,
\end{equation}
a long root subgroup. Hence $\D(t) = U$ in this case.

The proof of (ii) for $G = \PSU_n(q)$ is very similar. Take $V$ to be an $n$-dimensional unitary space over $\F_{q^2}$, with unitary form having Gram matrix with 1's on the reverse diagonal and 0's elsewhere. Then we can take $t$ to be as in (\ref{rep}) with $\l+\l^q = 0$, and we compute as above that the only elements $u \in C_G(t)$ for which 
both $u$ and $tu^{-1}$ are conjugate to $t$ lie in a subgroup $U$ defined as in (\ref{ugp}), but with the condition on the scalar $c \in \F_{q^2}$ being $c+c^q=0$. Then $U$ is the centre of a long root subgroup of $G$, and $\D(t)=U$, proving (ii) in this case. 

Now let $G = \,^2\!G_2(q)$ or $^2\!B_2(q)$ as in (iii), with $q = p^{2a+1}$, where $p=2$ or 3, and $a\ge 1$. Then $C_G(t) = P$, a Sylow $p$-subgroup of $G$. The structure and fusion of $P$ is described in \cite{Suzuki, Ward}: $Z(P)$ is elementary abelian of order $q$ and has all its non-identity elements conjugate to $t$; also there is no element $u \in P\setminus Z(P)$ such that $u$ and $tu^{-1}$ are conjugate to $t$. Hence $\D(t) = Z(P)$ in these cases, completing the proof.
  \end{proof}

We now prove Theorem \ref{t: component} separately for the various Lie types.

\subsection{Linear groups}

\begin{lem}\label{psl} Theorem $\ref{t: component}$ holds for $G = \PSL_n(q)$, $q = 2^a$.
\end{lem}

\begin{proof} Let $G = \PSL_n(q)$ with $q=2^a$. If $n=2$, the only involution class in $G$ contains a root element $t = J_2 \in \SL_2(q) = \Sp_2(q)$, and this class is covered by Lemma \ref{rooty}.  

So assume $n\ge 3$. 
We begin by establishing the result for $n\le 4$. For $n=3$, the group $\PSL_3(q)$ has only one class of involutions, with representative $t = (J_2,J_1)$, a long root element, so this case is covered by Lemma \ref{rooty}. Now let $n=4$, $G = \PSL_4(q)$. Again, root elements $(J_2,J_1^2)$ are covered by Lemma \ref{rooty}, so let $t = (J_2^2) \in G$. We can take $t = J_2\otimes I \in \SL_2(q) \otimes \SL_2(q) < G$. Then $C_G(t)$ contains $I \otimes \SL_2(q)$, and each involution $u$ in this group is conjugate to $t$; also $tu = J_2\otimes J_2$ is conjugate to $t$. Hence $t$ is joined in $\G(\mathcal C)$ to all such involutions $u$. It follows that $\D(t)$ contains $I \otimes \SL_2(q)$. Similarly, arguing with neighbours of $u$, we see that $\SL_2(q) \otimes I \le \D(t)$. Thus
\[
\D(t) \ge \SL_2(q) \otimes \SL_2(q).
\]
Also, if we write $t$ as $\begin{pmatrix} I & I \\ 0 & I \end{pmatrix}$, then as in the proof of Lemma \ref{l: small classical}, we see that $\D(t)$ contains $N:=O_2(C_G(t)) \cong q^4$. Using \cite[Tables 8.8,8.9]{bhr}, we see that for $q \ge 4$, the only overgroups in $G$ of $\SL_2(q) \otimes \SL_2(q) = \O_4^+(q)$ are $O_4^+(q)$ and $\Sp_4(q)$, and hence $\la \SL_2(q) \otimes \SL_2(q) ,\,N \ra = G$. Finally, for $q=2$ we argue in $A_8 \cong L_4(2)$ that $\D(t) = G$ here too. This completes the proof for $n=4$.
Now suppose $n\ge 5$. Adopt the following notation, for $\l \in \F_q$:
\[
J_2(\l) = \begin{pmatrix} 1 & \l \\ 0 & 1 \end{pmatrix}, \;J_2 = J_2(1),\;, t = (J_2^a,J_1^b),
\]
where $2a+b=n$. By Lemma \ref{rooty} we can assume that $a\ge 2$.

Assume next that $q>2$, and let $\a \in \F_q \setminus \{0,1\}$. Suppose $b\ge 1$. Write $t = t_1\oplus t_2$, where $t_1 =(J_2,J_1)$, $t_2 = (J_2^{a-1},J_1^{b-1})$. Let $u_2 = (J_2(\a)^{a-1},J_1^{b-1})$. If $u_1 \in \SL_3(q)$ is an involution joined to $t_1$ in the graph on $t_1^{SL_3(q)}$, then $t$ is joined to $u = u_1\oplus u_2$ in the graph on $t^G$. Hence by the result for $\SL_3(q)$, we have $\D(t) \ge \SL_3(q)$. This holds for any choice of blocks $(J_2,J_1)$ in $t$, and these $\SL_3(q)$ subgroups generate $G$. So $\D(t) = G$ in the case where $b\ge 1$.

Now suppose $b=0$ (still assuming that $q>2$). For this case, write  $t = t_1\oplus t_2$, where $t_1 =(J_2^2)$, $t_2 = (J_2^{a-2})$. Let $u_2 = (J_2(\a)^{a-2})$, and argue as above using the result for $\SL_4(q)$ that $\D(t)$ contains the subgroup $\SL_4(q)$ corresponding to $t_1$. This holds for any choice of blocks $J_2^2$ in $t$, so again we see that $\D(t) = G$. 

It remains to deal with the case where $q=2$. It is convenient first to deal with $n=5,6$. For these cases we have $t = (J_2^2,J_1^{n-4})$ (excluding the exceptional case $(J_2^3) \in \SL_6(2)$ as it is conclusion (iii) of Theorem \ref{t: component}). 
For $n=5$, we write $t = t_1\oplus t_2$, where $t_1 = (J_2^2)$, $t_2 = J_1$. Arguing as above using the result for $\SL_4(q)$, we see that $\D(t)$ contains $S:=SL_4(q)$. Also $\D(t)$ is $C_G(t)$-invariant, and we can compute a $C_G(t)$-conjugate $S^c$ of $S$ such that $\la S,S^c\ra = G$. Thus $\D(t) = G$ in the case $n=5$. For $n=6$ we argue similarly, taking 
$t = t_1\oplus t_2$, where $t_1 = (J_2,J_1)$, $t_2 = (J_2,J_1)$ and using the result for $\SL_3(q)$.

Finally, assume that $n\ge 7$ (with $q=2$). Write $t = t_1\oplus t_2$, where 
\[
t_1 = \left\{\begin{array}{l} (J_2,J_1),\hbox{ if }a=2  \\ (J_2^2), \hbox{ if }a\ge 3 \end{array} \right.
\]
Then $t_1,t_2$ are involutions in $\SL_{n_1}(2)$, $\SL_{n_2}(2)$, where $n_1=3$ or 4, $n_2 = n-n_1$. If 
$u_i \in \SL_{n_i}(2)$ ($i=1,2$) are involutions joined to $t_i$ in the graphs on $t_i^{SL_{n_i}(2)}$, then $u = u_1\oplus u_2$ is joined to $t$ in $\G(\mathcal C)$. Hence inductively we have $\D(t) \ge \SL_{n_1}(2) \times \SL_{n_2}(2)$ (the second factor possibly replaced by $2^9$ in the case where $n=10$, $t = (J_2^5)$). By re-ordering the blocks in $t$, we obtain several different such subgroups in $\D(t)$, and these generate $G$. This completes the proof. 
\end{proof}

\subsection{Unitary groups}\label{unitary}

\begin{lem}\label{psu} Theorem $\ref{t: component}$ holds for $G = \PSU_n(q)$ with $n\ge 3$,  $q = 2^a$.
\end{lem}

\begin{proof} Let $G = \PSU_n(q)$ with $q=2^a$. If $n=3$ then there is only one class of involutions, namely the long root elements, and these are covered by Lemma \ref{rooty}. So we may suppose that $n\ge 4$, and also that $t \in G$ is an involution that is not a long root element. 

If $n=4$ then $t = (J_2^2)$ and we argue as in the proof of Lemma \ref{psl} (second paragraph) that $\D(t) = G$, noting at the end that the case $q=2$ is covered by Lemma \ref{l: small classical}. 

Now assume that $n\ge 5$ and $q>2$. Let $t = (J_2^a,J_1^b)$ (with $a\ge 2$, as $t$ is not a long root element). We argue by induction on $n$. Write $t = t_1\oplus t_2$, where $t_1 = (J_2^a,J_1^{b-1})$ if $b\ge 1$, and $t_1 = (J_2^{a-1})$ if $b=0$. Then inductively, $\D(t_1) = \SU_{n-1}(q)$ or $\SU_{n-2}(q)$ in the respective cases. Moreover $\D(t)$ contains this subgroup: in the first case this is clear, and in the second we can use neighbours of $t$ of the form $u_1\oplus u_2$, where $u_2 = J_2(\a)$, as in the proof of Lemma \ref{psl}. As $\D(t)$ is $C_G(t)$-invariant, it follows that $\D(t)=G$.

Finally, consider the case $n\ge 5$, $q=2$. By Lemma \ref{l: small classical}, we may ssume that $n\ge 9$. Let $t = (J_2^a,J_1^b)$ with $a\ge 2$. If $b\ge 1$, write $t = t_1\oplus t_2$, where $t_1 = (J_2^a,J_1^{b-1})$, and argue inductively as above (using Lemma \ref{l: small classical} for the base case where $t_1 \in \SU_8(2)$). Now suppose $b=0$, and write $t = t_1\oplus t_2$, where $t_1 = (J_2^2)$, $t_2 = (J_2^{a-2})$. If $a>5$, then inductively we have $\D(t) \ge \SU_4(2) \times \SU_{n-4}(2)$; also $\D(t)$ contains several subgroups of this form corresponding to different pairs of $J_2$-blocks, and these generate $G$. and if $a=5$, then according to the entry for $\PSU_6(2)$ in Table \ref{tbl}, $\D(t_2) = 2^9$, and so we see that $\D(t)$ contains $\D(t_1) \times \D(t_2) = \SU_4(2)\times 2^9$. The collection of such subgroups $\SU_4(2)$, one for each pair of $J_2$-blocks in $t$, generates $G$, and hence $\D(t) = G$ in this case, completing the proof. 
\end{proof}

\subsection{Symplectic and Orthogonal groups}\label{orth}

In this section we consider the symplectic and orthogonal groups $ G = \Sp_{2n}(q)$, $\O_{2n}^\e(q)$ with $q = 2^a$ and $\e = \pm$. Note that $O_{2n}^\e(q) < \Sp_{2n}(q)$, and let $V = V_{2n}(q)$ be the natural module. We begin by describing the involution classes in $G$, with notation and results taken from \cite[Chapters 4-6]{lieseitz3}. There is an involution in $O_2^\e (q)\setminus \O_2^\e(q)$ (hence also in $\Sp_2(q)$), which we denote by $V(2)$. Also the group $\O_4^+(q)$ has a subgroup $\SL_2(q)$ stabilizing a pair of totally singular 2-spaces, and we denote an involution in this $\SL_2$ subgroup by $W(2)$; it acts as $J_2^2$ on the 4-dimensional space. Denote by $W(1)$ the identity element of $\O_2^\e(q)$. Then, with one exceptional case,  every involution $t \in G$ is uniquely determined up to conjugacy by an orthogonal decomposition
\begin{equation}\label{eq}
t = W(1)^a + W(2)^b + V(2)^c,
\end{equation}
where $c \le 2$  (and $n = a+2b+c$). The exceptional case is  $t = W(2)^b \in \O_{4b}^+(q)$, in which case there are two $G$-classes of involutions which are interchanged by elements of $O_{4b}^+(q)\setminus \O_{4b}^+(q)$. 

Note that in (\ref{eq}), $c$ can only be 1 if $G = \Sp_{2n}(q)$. Also, when $G$ is orthogonal and $c=2$, there are actually two possibilities for the blocks $V(2)^2$ -- in the notation of \cite{GLO}, they could be $V(2)+V(2) \in \O^+_4(q)$ or $V(2)+V_\a(2) \in \O_4^-(q)$ (where $\a \in \F_q$ is such that the quadratic $x^2+x+\a$ is irreducible over $\F_q$). For our proof below it is not necessary to distinguish between these cases, so we use the notation $V(2)^2$ for both.

\begin{lem}\label{sp} Theorem $\ref{t: component}$ holds for $G = \Sp_{2n}(q)$ with $n\ge 2$,  $q = 2^a$.
\end{lem}

\begin{proof} Let $G = \Sp_{2n}(q)$ with $q=2^a$. Note that the class of long root elements of $G$ is represented by $t = V(2)+W(1)^{n-1}$, and the conclusion of Theorem \ref{t: component} for this class follows from Lemma \ref{rooty}. Henceforth we do not consider this class.

Suppose first that $n=2$ and $q>2$. The classes of long and short root elements of $G$ are interchanged by a graph automorphism , so the conclusion of the theorem for short root elements follows from Lemma \ref{rooty}. The remaining  involution class in $G$ contains $t = V(2)^2$. This element lies in a subgroup $\Sp_4(2)' \cong A_6$ of $G$, and arguing in $A_6$ we see that $\D(t) \ge A_6$. As $\D(t)$ is $C_G(t)$-invariant, it follows that $\D(t) = G$ for this class.

Next consider $n=3$. If $q=2$ then $G = \Sp_6(2)$, which is covered by Lemma \ref{l: small classical}, so assume $q>2$. The classes to consider are those containing $t = W(2)+W(1)$, $V(2)^2+W(1)$ and $W(2)+V(2)$. In the first two cases $t$ lies in a subgroup $\Sp_6(2)$, so Lemma \ref{l: small classical} shows that $\D(t) \ge \Sp_6(2)$, and now the $C_G(t)$-invariance of $\D(t)$ shows that it is equal to $G$. In the last case, $t$ is conjugate to $V(2)^3$, and writing $t = t_1\oplus t_2$ with $t_1 = V(2)^2 \in \Sp_4(q)$, we see that $\D(t)$ contains $\D(t_1) = \Sp_4(q)$ and hence that $\D(t) = G$. 

If $n = 4$ or 5 then since any representative $t$ as in (\ref{eq}) lies in a subgroup $\Sp_{2n}(2)$ of $G$, Lemma \ref{l: small classical} shows that $\D(t)$ contains $\Sp_{2n}(2)$, and then $C_G(t)$-invariance gives $\D(t) = G$. The same argument applies for $n=6$, except for the class $t = W(2)^3$ (the exceptional class in $\Sp_{12}(2)$ in Table \ref{tbl}). In this case, Lemma \ref{tbl} gives the conclusion if $q=2$, and for $q>2$ we write $t = t_1 \oplus t_2$ with $t_1 = W(2)^2$, $t_2=W(2)$ and see that $\D(t)$ contains $\D(t_1) = \Sp_8(q)$, hence that $\D(t) = G$.

Suppose finally that $n\ge 7$. Let $t \in G$ be as in (\ref{eq}). We argue by induction on $n$, having established the base cases $n=3,4,5,6$. 

Assume $a\ge 1$, and write $t = t_1\oplus t_2$ with $t_1 = W(1)^{a-1} + W(2)^b + V(2)^c$, $t_2 = W(1)$. Then $\D(t)$ contains $\D(t_1)$, which inductively is $\Sp_{2n-2}(q)$, unless $t_1 = W(2)^3 \in \Sp_{12}(2)$ and $G = \Sp_{14}(2)$. In the former case, $\D(t) = G$ in the usual way; in the latter, rewrite $t$ as $(W(2)^2)\oplus (W(2)+W(1)) \in \Sp_8(2) \times \Sp_6(2)$ - then Lemma \ref{l: small classical} gives $\D(t) \ge \Sp_8(2) \times \Sp_6(2)$, and now we see that $\D(t) = G$ as usual.

So we may assume $a=0$. Then $b\ge 3$, and we can write $t = t_1\oplus t_2$ with $t_1=W(2)^2$, $t_2 = W(2)^{b-2}+V(2)^c$. Then $\D(t)$ contains $\D(t_1) = \Sp_8(q)$, and this is the case for any pair of $W(2)$-blocks making up $t_1$. Hence $\D(t) \ge \Sp_{2n-2c}(q)$, and then $\D(t) = G$ in the usual way.
\end{proof}

\begin{lem}\label{orthrest} Theorem $\ref{t: component}$ holds for $G =  \O_{2n}^\e(q)$ with $q = 2^a$, $n\ge 4$ and $\e = \pm$.
\end{lem}

\begin{proof} Let $t \in G$ be as in (\ref{eq}). 
The proof goes by induction on $n$, and runs along entirely similar lines to that of Lemma \ref{sp}. We first establish the bases cases $n=4,5,6$ in exactly the same way, using Lemma \ref{l: small classical} to see that $\D(t)$ contains a subgroup $\O_{2n}^\d(2)$ and hence that $\D(t) = G$, in all cases except $t = W(2)^3 \in \O_{12}^+(q)$. For the latter class, the case $q=2$ is covered by Lemma \ref{l: small classical}, and for $q>2$ we write $t = t_1 \oplus t_2$ with $t_1 = W(2)^2$, $t_2=W(2)$; hence
$\D(t) \ge\D(t_1) = \Sp_8(q)$, from which we deduce that $\D(t) = G$.

Finally, for $n\ge 7$ we argue by induction exactly as in the proof of Lemma \ref{sp}.
\end{proof}

\subsection{Exceptional groups of Lie type}

\begin{lem}\label{excep} Theorem $\ref{t: component}$ holds for $G=G(q)$, a simple group of exceptional Lie type, where $q = 2^a$.
\end{lem}

Note that we exclude $G_2(2)'$ in the hypothesis.

\begin{proof}
First we consider the simple groups $G(q)$ of untwisted Lie type. A convenient list of conjugacy classes of involutions in these groups $G(q)$ can be found in the tables in \cite[Chapter 22]{lieseitz3}. Class representatives are products of involutions in Levi subgroups $A_1^k$, as listed in the tables. 

\vspace{2mm}
\no {\bf Case $G = E_8(q)$. }  There are four involution classes, corresponding to Levi subgroups $A_1^k$ for $k=1,2,3,4$. For $k=1$ this is a root involution, covered by Lemma \ref{rooty}. 

For $k=2$, we can take as representative $t = u_{\a_4}(1)u_{\a_8}(1)$ with the usual labelling of the Dynkin diagram. Then $t$ lies in a subsytem subgroup $A_7(q)$ generated by $U_{\pm \a_i}$ for $i=1,3,4,5,6,7,8$; it also lies in a subsystem $A_6(q)$ generated by $U_{\pm \a_i}$ for $i=2,4,5,6,7,8$. By Lemma \ref{psl}, $\D(t)$ contains both of these subsystem subgroups. Together they generate $G$, so $\D(t) = G$.

Now consider $k=3$, so $t \in A_1^3$. Here we can take $t =  u_{\a_4}(1) u_{\a_6}(1) u_{\a_8}(1)$, and we see  that $\D(t)$ contains subsystem subgroups $A_7(q)$ and $A_6(q)$ as in the previous case.

Finally consider $k=4$. A Levi subgroup $A_1^4$ lies in a subsystem subgroup $A_8$, so Lemma \ref{psl} gives $\D(t) \ge A_8(q)$, which is maximal. Also $\D(t)$ is $C_G(t)$-invariant, and $C_G(t) = [q^{84}].C_4(q)$ by \cite{lieseitz3}. Hence $\D(t) = G$.

\vspace{2mm}
\no {\bf Case $G = E_7(q)$. } Apart from root elements, there are four involution classes, labelled $A_1^2$, $(A_1^3)^{(1)}$, 
$(A_1^3)^{(2)}$, $A_1^4$, with representatives as follows:
\[
\begin{array}{l}
A_1^2: u_{\a_4}(1)u_{\a_7}(1) \\
(A_1^3)^{(1)}: u_{\a_2}(1) u_{\a_5}(1) u_{\a_7}(1) \\
(A_1^3)^{(2)}: u_{\a_3}(1) u_{\a_5}(1) u_{\a_7}(1) \\
A_1^4: u_{\a_2}(1)u_{\a_3}(1) u_{\a_5}(1) u_{\a_7}(1) 
\end{array}
\]
The $A_1^2$ class is dealt with as for $E_8(q)$. For $t$ in one of the $A_1^3$ classes, $t$ lies in a subsystem subgroup $D_6(q)$, so by Lemma \ref{orthrest} we have $\D(t)\ge D_6(q)$; one can adjoin roots to those defining the $A_1^3$ Levis to obtain several different $D_6$ subsystem subgroups, all of which lie in $\D(t)$ and generate $G$. Finally, for $t$ in the $A_1^4$  class, $t$ again lies in a subsytem $D_6(q)$, so $\D(t)$ contains $D_6(q)$ and is $C_G(t)$-invariant (here $C_G(t) = [q^{42}].C_3(q)$), which is enough to force $\D(t) = G$.

\vspace{2mm}
\no {\bf Cases $G = E_6(q)$, $F_4(q)$. } These are dealt with in similar fashion -- the class representatives $t$ lie in subsystem subgroups $D_5(q)$ or $D_4(q)$ respectively, so $\D(t)$ contains this subgroup and is $C_G(t)$-invariant, which forces $\D(t) = G$.

\vspace{2mm}
\no {\bf Case $G = G_2(q)$. } Here $q\ge 4$ (as we excluded $G_2(2)'$ from the hypothesis). Apart from long root involutions, there is one involution class $\tilde A_1$, with representative $t=u_{\a}(1)$, where $\a$ is a short root. Take $\a = \a_1+2\a_2$, where $\a_1,\a_2$ are simple roots with $\a_2$ short, and let $\Phi$ be the root system. 
Now $\tilde A_1 = \la U_{\pm \a}\ra$ is centralized by $A_1 = \la U_{\pm \a_1}\ra$, and the action of $A_1\tilde A_1$ on the 6-dimensional module $V_6=V_G(\l_1)$ is $1\otimes 1 \oplus 0\otimes 1^{(2)}$ (where 1 denotes the natural 2-dimensional module -- see \cite[Chapter 11]{lieseitz3}). In particular the long and short root involutions act on $V_6$ as $(J_2^2,J_1^2)$ and $(J_2^3)$, respectively. 
For $\l \in \F_q\setminus\{0,1\}$, and for any involution $x \in A_1$, the element $u = xu_{\a}(\l)$ is conjugate to both $t$ and $tu$; so $u$ is joined to $t$ in the graph $\G(\mathcal C)$. Hence $\D(t)$ contains $A_1\times U_{\a}$. It is also $C_G(t)$-invariant, and $C_G(t) = [q^3].A_1(q)$, from which we see that $\D(t) \ge C_G(t)$. Since $u$ is joined to $t$, also $\D(t) \ge C_G(u)$. We need to finally check that $\la C_G(t),C_G(u)\ra = G$. We shall prove this with $x = u_{\a_1}(1)$ (so 
$u = u_{\a_1}(1)u_{\a}(\l)$).

Suppose that $\la C_G(t),C_G(u)\ra \ne G$, and let $M$ be a maximal subgroup of $G$ containing $\la C_G(t),C_G(u)\ra $. From the maximal subgroups of $G$ (see for example \cite{bhr}), we see that $M$ must be the parabolic subgroup $P = \la U_{\pm \a_1}, \,U_\b \; : \b \in \Phi^+\ra$.

Write $U = \la U_\b : \b \in \Phi^+\ra$. Since $C_G(u) \le P$,  there exists $v \in U$ such that $A_1^{v^{-1}} \le C_G(u)$, and so $u^v \in C_G(A_1) = \tilde A_1$. Since $u^v \in U$, it follows that $u^v \in \tilde A_1 \cap U = U_\a$. However we can see from the commutator relations in $U$ that this is not possible. 

Therefore  $\la C_G(t),C_G(u)\ra = G$, completing the proof for $G_2(q)$.

\vspace{2mm}
We now move on to the twisted exceptional groups $G(q)$. For these, there are convenient lists of involution class representatives given in \cite{as}.

\vspace{2mm}
\no {\bf Case $G = \,^2\!E_6(q)$. } The involution class representatives $t$ lie in a subsystem subgroup $^2\!D_5(q)$, so by Lemma \ref{orthrest}, $\D(t)$ contains this subgroup and is $C_G(t)$-invariant, which forces $\D(t) = G$.

\vspace{2mm}
\no {\bf Case $G = \,^2\!F_4(q)'$. } For $q=2$, Lemma~\ref{l: small exceptional} implies that for both the involution class representatives $t \in \,^2\!F_4(2)'$ we have $\D(t) \ge \,^2\!F_4(2)'$.
Now assume $q\ge 4$. Note that $G$ has two classes of involutions. From \cite[p.75]{Atlas} we see that there are no involutions in $^2\!F_4(2)\setminus ^2\!F_4(2)'$. Hence both involution classes are represented by elements $t$ in a subgroup $^2\!F_4(2)'$ of $G$, and so by the $q=2$ case we have $\D(t) \ge \,^2\!F_4(2)'$. Also $\D(t)$ is $C_G(t)$-invariant, and it follows that $\D(t) = G$.

\vspace{2mm}
\no {\bf Case $G = \,^3\!D_4(q)$. } Again there are two involution classes, and both are represented by involutions $t$ in a subgroup $G_2(q)$ of $G$. Hence by the $G_2(q)$ case, for $q\ge 4$ we have $\D(t) \ge G_2(q)$, and so $\D(t) = G$ by the $C_G(t)$-invariance. Finally, for $q=2$ we have $\D(t) = G$ by Lemma~\ref{l: small exceptional}.

\vspace{2mm}
\no {\bf Case $G = \,^2\!B_2(q)$. } Here $q\ge 8$ and there is just one involution class $t^G$, and by Lemma \ref{rooty}, $\D(t)$ is the centre of a Sylow $2$-subgroup.

This completes the proof for exceptional groups.
\end{proof}

\section{Proof of Theorem \ref{t: char 2}}\label{s: proof char 2}

For the first part of this section $G=G(q)$ is a simple group of Lie type over $\Fq$, where $q=2^a$. We consider the binary action of $G$ on a set $\Omega$. We assume that $H$ is the stabilizer in $G$ of a point $\omega\in \Omega$ and that $H$ is a proper subgroup of $G$ of even order.

By \cite[Lemma~2.5]{ggAltBin} we know that the action of $G$ on $(G:H)$, the set of cosets of $H$ in $G$, is binary (or, put another way, we can assume that the action of $G$ is transitive). Since $H$ has even order, there exist involutions in $G$ that fix elements of $(G:H)$. We take $g$ to be an involution in $H$ of maximal 2-fixity for the action of $G$ on $(G:H)$. Now the following lemma immediately yields Theorem~\ref{t: char 2}:

\begin{lem}
 One of the following holds:
 \begin{itemize}
\item[{\rm (i)}] the element $g$ is a long root involution, $G = \Sp_{2n}(q)$, $\PSU_n(q)$ or ${^2\!B_2}(q)$, and $\D(g)$ is the centre of a long root subgroup that is contained in $H$ (the centre of a Sylow $2$-subgroup when $G = {^2\!B_2}(q)$);
\item[{\rm (ii)}] the element $g$ is a short root involution, $G = \Sp_{4}(q)$ with $q>2$, and $\D(g)$ is a short root subgroup that is contained in $H$. 
\end{itemize}
If $G={^2\!B_2}(q)$, then $H=\D(g)$.
\end{lem}

\begin{proof}
The four possibilities for $\Delta(g)$ are given by Theorem~\ref{t: component}. We examine these in turn, noting that Lemma~\ref{l: graph} implies that $\Delta(g)\leq H$.
 
If item (i) of Theorem~\ref{t: component} holds, then we have $H=G$ which is a contradiction. Items (ii) and (iii) are listed in the statement of the lemma.
 
Finally, if item (iv) of Theorem~\ref{t: component} holds, then $G, g$ and $\Delta(g)$ are shown in Table~\ref{tbl}. But, in this case, Lemma~\ref{l: terminal table} implies that $\Delta_\infty(g)=G$ and Lemma~\ref{l: terminal component} implies that $\Delta_\infty(g)\leq H$. Once again we have a contradiction.

To complete the proof of the lemma we use the fact that if $G={^2\!B_2}(q)$, then $G$ has a single conjugacy class of involutions and so the fact that $H=\D(g)$ follows from \cite[Theorem~1.1]{ggLieBin}.
\end{proof}

\section{A partial converse to Theorem~\ref{t: char 2}}\label{s: pc}

The main result of this section, Proposition~\ref{p: pc}, pertains to the groups listed at items (ii) and (iii) of Lemma~\ref{rooty}. We show that particular actions of these groups are binary. These are the first non-regular transitive binary actions to have been described for the families $\PSp_n(q)$, $\PSU_n(q)$ and $^2\!G_2(q)$.

As we mention in the introduction, we conjecture that the actions considered in Proposition~\ref{p: pc} include all of the transitive binary actions of a simple group of Lie type over a field of characteristic $2$, for which a point-stabilizer has even order.

Recall that a group $H$ is a \emph{TI-subgroup} of a group $G$ if $H$ is a subgroup of $G$ with the property that $H\cap H^g$ is either $H$ or $\{1\}$ for $g\in G$. Now we need the following result.

\begin{lem}\label{l: ti}\cite[Lemma~2.3]{ggLieBin}
 Suppose that $G$ acts on $\Omega$, the set of cosets of a TI-subgroup $H$. The action is binary if and only if, given any three distinct conjugates $H_1$, $H_2$ and $H_3$ of $H$, we have $H_1\cap (H_2 H_3)=\{1\}$.
\end{lem}

\begin{lem}\label{l: sl2}
Let $S=\SL_2(q)$ where $q=p^a$ for some prime $p$ and positive integer $a$. Let $H$ be a Sylow $p$-subgroup of $G$. For all $H_1, H_2$ and $H_3$, distinct conjugates of $H$, we have $H_1\cap (H_2 H_3)=\{1\}$.
\end{lem}
\begin{proof}
 Since $S$ acts 2-transitively on the set of Sylow $p$-subgroups of $S$, we can assume that $H_2$ is the set of strictly upper-triangular matrices and $H_3$ is the set of strictly lower-triangular matrices. Now consider $h_2\in H_2$ and $h_3\in H_3$ and observe that
\[
 h_2h_3=\begin{pmatrix}
  1 & a \\ 0 & 1
 \end{pmatrix}\begin{pmatrix}
 1 & 0 \\ b & 1 \end{pmatrix} = \begin{pmatrix} 1+ab & a \\ b & 1 \end{pmatrix}.
\]
We wish to determine if $h_2h_3$ can lie in a third conjugate, $H_1$, of $H$. If this were the case, then $h_2h_3$ would be an element of order $p$. This requires that $h_2h_3$ has two eigenvalues equal to $1$ and so the characteristic polynomial is $\lambda^2-2\lambda+1$. But this would imply that $a$ or $b$ is equal to $0$. Thus if we take $h_2$ and $h_3$ to be non-trivial, then $h_2h_3$ is not an element of order $p$. The result follows.
\end{proof}

Lemmas~\ref{l: ti} and \ref{l: sl2} show that the action of $\SL_2(q)$ on the cosets of a Sylow $p$-subgroup is binary. But in contrast, for $q$ odd the action of $\PSL_2(q)$ is not binary, as shown in \cite{ggLieBin}.

\begin{prop}\label{p: pc}
Let $G$ and $H$ be one of the following:
\begin{enumerate}
 \item $G = \PSp_{2n}(q)\,(n\ge 1)$ or $\PSU_n(q)\,(n\ge 3)$ and $H$ is the centre of a long root subgroup of $G$;
 \item  $G = \,^2\!G_2(q)$ or $^2\!B_2(q)$, and $H$ is the centre of a Sylow $p$-subgroup of $G$ with $p=3$ or $2$, respectively.
\end{enumerate}
Then the action of $G$ on the set of right cosets of $H$ is binary if and only if $G$ is not equal to $\PSp_2(q)$ with $q$ odd.
\end{prop}

Note that $q$ is an arbitrary prime power. When $q$ is even, this lemma gives a partial converse to Theorem~\ref{t: char 2}.

\begin{proof}
 It is well known, and easy to confirm directly, that in every case $H$ is a TI-subgroup of $G$. We will, therefore, prove Proposition~\ref{p: pc} using Lemma~\ref{l: ti}.

 If $G$ is as in item (1), then we suppose, first, that $G=\PSp_{2}(q)=\PSL_2(q)$. In this case the action of $G$ on the set of right cosets of $H$ is binary if and only if $q$ is even, by \cite[Theorem~1.2]{ggLieBin}.

Suppose next that $G=\PSU_3(q)$. Then the action of $G$ on the set of conjugates of $H$ is $2$-transitive and any distinct pair of conjugates of $H$ generate a subgroup isomorphic to $\SL_2(q)$. If $(H_2, H_3)$ is such a pair, then $H_2,H_3$ are Sylow $p$-subgroups of $\langle H_2, H_3\rangle$ and so, if $H_1$ is a third conjugate of $H$ satisfying $H_1\cap H_2 H_3\neq \{1\}$, then $H_1$ must be a subgroup of $\langle H_2, H_3\rangle$ (since $H$ is a TI-subgroup). But now Lemma~\ref{l: sl2} implies that in this case $H_1\cap H_2 H_3=\{1\}$. Thus the criterion of Lemma~\ref{l: ti} is satisfied and we conclude that the action is binary.

Suppose next that $G=\PSp_{2n}(q)$ with $n\geq 2$ or $\PSU_n(q)$ with $n\ge 4$. Then the action of $G$ on the set of conjugates of $H$ has rank $3$ (see \cite{kan_lieb}) and so $G$ has 3 orbits in its action by conjugation on the set of pairs of conjugates of $H$, one of which includes the pair $(H,H)$, which we can ignore.

A second orbit includes the  pair $(H, H^{opp})$ where $H^{opp}$ is the centre of an  `opposite' root subgroup of $H$; in particular $\langle H, H^{opp}\rangle \cong \SL_2(q)$. Now the argument proceeds just as for $\PSU_3(q)$, via Lemma~\ref{l: sl2}, to lead us to conclude that the criterion of Lemma~\ref{l: ti} is satisfied  when $(H_2, H_3)$ is in this orbit.

The third orbit contains all pairs of conjugates of $H$ that commute with each other. We take $G=\PSp_{2n}(q)$ first and display two such conjugates explicitly as follows: we let $\{e_1,\dots, e_n, f_1,\dots, f_n\}$ be a hyperbolic basis for $V=\Fq^n$ with respect to an alternating form $\varphi$ and we set $H_2^\dagger$ to be the long root subgroup consisting of  linear maps $g$ that fix every basis vector except for $e_1$ and for which there exists $\alpha_g \in \Fq$ such that
\[
 e_1^g = e_1+\alpha_gf_1.
\]
Similarly $H_3^\dagger$ is the long root subgroup consisting of  linear maps $g$ that fix every basis vector except for $e_2$ and for which there exists $\alpha_g \in \Fq$ such that
\[
 e_2^g = e_2+\alpha_gf_2.
\]
It is clear that non-trivial elements of $H_2^\dagger$ and $H_3^\dagger$ have a single Jordan block of size $2$, whereas an element of $H_2^\dagger H_3^\dagger\setminus (H_2^\dagger \cup H_3^\dagger)$ has two Jordan blocks of size $2$.

Writing $H_2$ (resp. $H_3$) for the projective image of $H_2^\dagger$ (resp. $H_3^\dagger$) we conclude that the only long root elements in $\langle H_2, H_3\rangle$ lie in $H_2\cup H_3$. Since $H$ is a TI-subgroup, we conclude that if $H_1$ is a third conjugate of $H$, then $H_1\cap H_2 H_3=\{1\}$ and the criterion of Lemma~\ref{l: ti} is verified in this case too. We conclude, therefore, that the action of $G$ on the set of right cosets of $H$ is binary.

The proof for $G=\PSU_n(q)$ is identical except that our basis is written with respect to a Hermitian form and we require that $\alpha_g \in \mathbb{F}_{q^2}$ satisfies $\alpha_g+\alpha_g^q=0$.

If $G$ is as in item (2), then the result for $G=\,^2\!B_2(q)$ follows from \cite[Theorem~1.1]{ggLieBin}. We are left with the case $G = \,^2\!G_2(q)$ for which we use facts from \cite{Ward}. The group $H$ is the centre of a Sylow 3-subgroup of $G$, and all of its non-identity elements lie in the same conjugacy class $C$; a representative is denoted by $X$ in \cite{Ward}.  Moreover $H \cong (\F_q,+)$, and we can parametrize the elements of $H$ as $\{h(c): c \in \F_q\}$, where $h(c)h(d) = h(c+d)$.

Let $H_2,H_3$ be distinct conjugates of $H$. We claim that
\begin{equation}\label{h2h3}
H_2H_3 \cap C = (H_2\cup H_3)\setminus \{1\},
\end{equation}
 from which the result will follow in the usual way from Lemma~\ref{l: ti}. To prove (\ref{h2h3}), we use character theory to count the number of triples in the set $S:= \{(x_1,x_2,x_3) \in C^3 : x_1x_2x_3=1\}$. By the well-known Frobenius formula (see for example \cite[30.4]{James-Liebeck}), we have
\[
|S| = \frac{|C|^3}{|G|} \sum_{\chi \in Irr(G)} \frac{\chi(X)^3}{\chi(1)}.
\]
All the values $\chi(X)$ for $\chi \in Irr(G)$ are listed in \cite[p.87]{Ward}, and we compute from these that
\[
|S| = (q^3+1)(q-1)(q-2).
\]
The number of triples in $S$ of the form $(h(c),h(d),h(-c-d))$ is $(q-1)(q-2)$, and if we multiply this by the number of conjugates of $H$, we obtain $|S|$. It follows that every triple in $S$ is a conjugate of such a triple $(h(c),h(d),h(-c-d))$, and equation (\ref{h2h3}) follows.
\end{proof}

\bibliography{myrefs}
\bibliographystyle{plain}

\end{document}